\documentclass{article}

\usepackage[utf8]{inputenc}

\usepackage{tikz}
\usepackage{float}
\usepackage{amssymb}
\usepackage{amsmath}

\usepackage{amsthm}
\newtheorem{theorem}{Theorem}[section]
\newtheorem{lemma}{Lemma}[section]
\theoremstyle{definition}
\newtheorem{definition}{Definition}[section]
\newtheorem{example}{Example}[section]

\title{Markov theorem for doodles on two-sphere}

\author{Konstantin Gotin}

\begin{document}

\maketitle

\begin{abstract}
In 1997 M.~Khovanov proved that any doodle can be presented as closure of twin, this result is analogue of classical Alexander's theorem for braids and links. We give a description of twins that have equivalent closures, this theorem is analogue of classical Markov theorem.
\end{abstract}

\section{Introduction}

Links can be considered as equivalence classes of planar diagrams up to Reidemeister moves. Classical Markov's and Alexander's theorems allow us to consider links as classes of braids, see \cite{1} for details. In this paper we consider doodles on $S^2$ and prove the analogue of Markov theorem for them.

Originally doodles were introduced by R. Fenn and P. Taylor in \cite{2} as a collection of piecewise-linear closed curves on a two-sphere $S^2$ without triple or higher intersection points. Later M.~Khovanov in \cite{3} offered to consider any component as an immersed circle on $S^2$. A.~Bartholomew, R. Fenn, N.~Kamada, S.~Kamada in \cite{4} generalize the notion of a doodle to be a collection of immersed circles in closed oriented surfaces of arbitrary genus and introduced the virtual doodles.

Also M.~Khovanov in \cite{3} proved that any doodle can be presented in the special form as closure of twin diagram.
The notion of a doodle is close to notion of a classical link and M.~Khovanov proved analogue of Alexander's theorem for doodles, here twin groups play the same role as braid groups in the classical knot theory.

In Section \ref{sec-2} we give some necessary definitions, most of them have analogues in the classical knot theory. For example, smoothing, bending and tightening of diagrams are defined. Section \ref{sec-3} is devoted to of sequences that relate diagram with correspondent minimal diagram and the study of their properties, see theorems \ref{th_circles_nuber}, \ref{th_bend_seq}, \ref{th_gen_bend_uniq} and \ref{th_uniq_annular}. And in Section \ref{sec-4} we introduce the equivalence relation for twins and prove Theorem \ref{th_markov}, the analogue of Markov theorem for doodles.

Now the studying of twin groups is concentrated on its algebraic properties, for example, in \cite{5} some properties of commutator subgroups were investigated. The Theorem \ref{th_markov} allows us to study twin groups for understanding the structure and classification of doodles.

\section{Basic definitions}\label{sec-2}
Main propose of this paper is to prove the analogue of the classical Markov theorem. We will define doodles and twins using classical language of diagrams to draw a strong analogy with classical knot theory.

\begin{definition}
\emph{Doodle diagram} with $m$ components is an immersion of $m$ disjoint circles to $S^2$ with no triple or higher intersection points. We assume that number of double points of doodle diagram is finite.
\end{definition}
Two doodle diagrams are said to be \emph{equivalent} if they can be related by a finite sequence of local moves $R_1$ and $R_2$ shown in Figure \ref{r1r2} and isotopies of $S^2$ (assume $R_1$ and $R_2$ to be orientation preserving if it is given).

\begin{figure}[H]
\begin{center}
\begin{tikzpicture}[scale = 0.4]
\draw[line width=1.4pt] plot [smooth, tension=0.7] coordinates 
{(0,0)(2.8,5)(3.8,3)(2.8,1)(0,6)};
\draw [line width=1.7pt, <->] (4.5,3) -- (6.2, 3);
\draw [line width=1.4pt] (7.5,0) -- (7.5,6);
\end{tikzpicture}
\hspace{6em}
\begin{tikzpicture}[scale = 0.4]
\draw[line width=1.4pt]  plot [smooth, tension=0.7] coordinates
{(0,0)(2,2)(2,4)(0,6)};
\draw[line width=1.4pt]  plot [smooth, tension=0.7] coordinates
{(3,0)(1,2)(1,4)(3,6)};
\draw [line width=1.7pt, <->] (3,3) -- (4.7, 3);
\draw[line width=1.4pt]  plot [smooth, tension=0.7] coordinates
{(7.2,0)(7,2)(7,4)(7.2,6)};
\draw[line width=1.4pt]  plot [smooth, tension=0.7] coordinates
{(5.2,0)(5.4,2)(5.4,4)(5.2,6)};
\end{tikzpicture}
\end{center}
\caption{The moves $R_1$ and $R_2$.}\label{r1r2}
\end{figure}
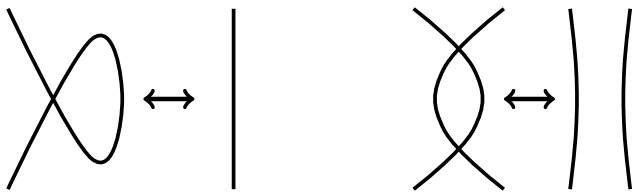

\begin{definition}
\emph{Doodle} is equivalence class of doodle diagrams. An oriented doodle is a doodle with an orientation of each component.
\end{definition}

Take a plane $\mathbb{R}^2 = \{ (x,y) : x,y \in \mathbb{R}\}$ and fix points $Q_i = (i,1)$ and $P_i = (i,0)$ for $i = 1\ldots,n$. 
\begin{definition}
A \emph{twin diagram} on $n$ strands is a configuration of $n$ arcs in $\mathbb{R} \times [0,1]$ such that:
\begin{enumerate}
\item for any $i = 1,\ldots,n$ there is unique $j = 1,\ldots,n$ such that $Q_i$ and $P_j$ are connected by a curve,\label{item1}
\item any curve is  monotonic by $y$-coordinate,\label{item2}
\item the number of double points is finite and there no triple or higher intersection points.\label{item3}
\end{enumerate}
We assume that any curve is oriented along $y$-coordinate.
\end{definition}

Two twin diagrams are \emph{equivalent} if they can be related by a finite sequence of moves $R_2$ and isotopies of $\mathbb{R} \times (0,1)$ such that conditions (\ref{item1}), (\ref{item2}), (\ref{item3}) are satisfied. 

\begin{definition}
A \emph{twin} on $n$ strands is an equivalence class of twin diagrams on $n$ strands.
\end{definition}

The product of two twins $\tau_1$ and $\tau_2$ on the same number of strands is defined by putting diagram of $\tau_1$ on top of the diagram of $\tau_2$ and squeezing along $y$-coordinate. It is easy to see that this product is well-defined and turns the set of twins on $n$ strands into a group denoted by $TW_n$. The unit element is represented by the diagram without double points. We will call such diagram the \emph{trivial} twin diagram.

\begin{theorem}[M.~Khovanov, \cite{3}]
A group $TW_n$ is generated by elements $s_1, s_2, \ldots, s_{n-1}$ presented in Figure \ref{pic_pi} which satisfy the following relations: 
\begin{center}
$s_i^{2}=e, \quad \mbox{\rm for } i = 1,\dots,n-1$, \\
$s_is_j=s_js_i, \quad \mbox{\rm if } |i-j|>1$
\end{center}

\begin{figure}[!ht]
\begin{center}
\begin{tikzpicture}[scale = 0.6]
\draw[line width=1.4pt] (0.5,1) -- (0.5,4);
\draw[line width=1.4pt] (2.7,1) to[in=-90, out=90] (4.3,4);
\draw[line width=1.4pt] (4.3,1) to[in=-90, out=90] (2.7,4);
\draw[line width=1.4pt] (6.5,1) -- (6.5,4);
\draw[line width=1.4pt] (2.1,1) -- (2.1,4);
\draw[line width=1.4pt] (4.9,1) -- (4.9,4);
\draw (2.7,0.5) node {$i$};
\draw (4.3,0.5) node {$i+1$};
\draw (0.5,0.5) node {$1$};
\draw (6.5,0.5) node {$n$};
\draw (1.25,2.5) node {$\ldots$};
\draw (5.75,2.5) node {$\ldots$};
\end{tikzpicture}
\end{center}
\caption{Diagram of generator $p_i$.}\label{pic_pi}
\end{figure}
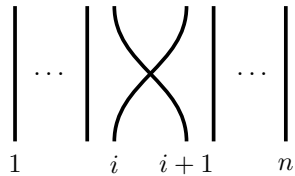
\end{theorem}

\begin{definition}
For a twins $\alpha$ and $\beta$ denote by $\alpha \otimes \beta$ the twin which diagram can be presented by adding diagram of $\alpha$ to diagram of $\beta$ from the left. 
\end{definition}

Further we consider $S^2$ presented by $\mathbb{R}^2 \cup \{\infty\}$ and consider the twin diagrams as diagrams on two-sphere. The following two definitions give a correspondence between doodle diagram and twin diagram and correspondence between doodle and twin.

\begin{definition}
Let $d$ be a twin diagram on $n$ strands.
The \emph{closure} of $d$ denoted by $\widehat{d}$ is a diagram obtained from $d$ by joining $P_i$ and $Q_i$ without adding new double points for $i = 1,\ldots,n$. Orientation of $d$ induces orientation of its closure.
\end{definition}

\begin{definition}
Let $d$ be a diagram of a twin $\beta$.
The \emph{closure} $\widehat{\beta}$ of $\beta$ is the doodle corresponding to diagram $\widehat{d}$. Obviously, closure of a twin is well-defined.
\end{definition}

The following theorem is some analogue of classical Alexander's theorem for links.
\begin{theorem}[M.~Khovanov, \cite{3}]\label{th_kh1}
Every oriented doodle on a two-sphere is the closure of a twin. 
\end{theorem}

And other result states the important property of doodles namely the existence and uniqueness of so-called minimal doodle. 

\begin{theorem}[M.~Khovanov, \cite{3}]\label{th_kh2}
A doodle has a unique (up to the transformation in Figure \ref{pic_circle-shift}) diagram with a minimal number of double points. This diagram can be constructed from any other doodle diagram by applying only those moves $R_1$ and $R_2$ that reduce the number of double points.
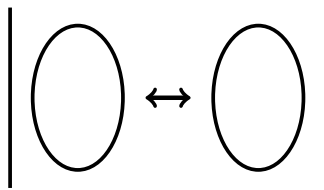
\begin{figure}[H]
\begin{center}
\begin{tikzpicture}[scale = 0.6]
\draw[line width=1.4pt] (0,0) -- (0,4);
\draw[line width=1.4pt] (1.5,2) circle [x radius=1cm, y radius=1.6cm];
\draw[line width=1.4pt] (7,0) -- (7,4);
\draw[line width=1.4pt] (5.5,2) circle [x radius=1cm, y radius=1.6cm];
\draw[line width=1.7pt, <->] (3,2) -- (4, 2);
\end{tikzpicture}
\end{center}
\caption{Circles shift.}\label{pic_circle-shift}
\end{figure}
\end{theorem}

Following definition generalize the notion of closure of twin diagram. We omit conditions for double points to be located on $\mathbb{R} \times (0,1)$.
\begin{definition}
Doodle diagram $D$ on $S^2$ is said to be \emph{annular diagram} if there is isotopy of $S^2$, connecting $D$ and twin closure $\widehat{\beta}$ for some twin $\beta$. We will consider annular diagrams up to isotopies of $S^2$.
\end{definition}

\begin{definition}
The transformation defined in Figure \ref{smooth} we will call the \emph{smoothing} of double point.
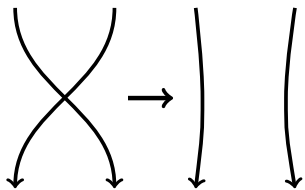
\begin{figure}[H]
\begin{center}
\begin{tikzpicture}[scale = 0.6]
\draw[line width=1.4pt, <-] (1,0) to[in=-90, out=90] (3.2,4);
\draw[line width=1.4pt, <-] (3.2,0) to[in=-90, out=90] (1,4);
\draw [line width=1.7pt, ->] (3.5,2) -- (4.5, 2);
\draw[line width=1.4pt, <-]  plot [smooth, tension=0.7] coordinates
{(7.2,0)(7,2)(7.2,4)};
\draw[line width=1.4pt, <-]  plot [smooth, tension=0.7] coordinates
{(5,0)(5.15,2)(5,4)};
\end{tikzpicture}
\end{center}
\caption{Smoothing of double point.}\label{smooth}
\end{figure}
\end{definition}

\begin{definition}
The result of smoothing of all double points in oriented doodle diagram $D$ is a collection of finite number of disjoint oriented simple curves. This curves are called \emph{the Seifert circles} of $D$.
\end{definition}

Two arcs of doodle diagram $D$ are belong to different Seifert circles if they are belong to different Seifert circles after smoothing all double points in $D$. Otherwise two arcs belong to the same Seifert circle. Collection of circles on the two-sphere is concentric if it can be deformed to diagram of closure of trivial twin by isotopy of $S^2$.

\begin{lemma}\label{smooth-lemma}
Let $D$ be an angular doodle diagram. The result of smoothing of all double points in $D$ is a collection of concentric circles with the same orientation. 
\end{lemma}
\begin{proof}
Since smoothing is invariant under isotopy, we can assume that $D = \widehat{d}$, for some twin diagram $d$.
Result of smoothing of all double points in $d$ is the trivial twin diagram. Hence the statement of the lemma is obvious.
\end{proof}

\begin{definition}
The following fragments of diagram presented on Figure \ref{pic_biangles} are called \emph{regular} (on the left) and \emph{irregular} (on the right) biangles.
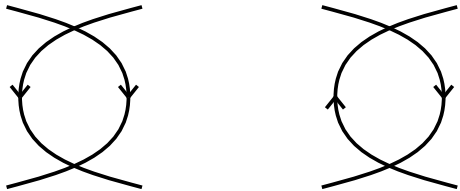
\begin{figure}[H]
\begin{center}
\begin{tikzpicture}[scale = 0.6]
\draw[line width=1.4pt]  (0,0) to[in=-90,out=15] (2.7,2);
\draw[line width=1.4pt]  (2.7,2) to[in=-15,out=90] (0,4);
\draw[line width=1.4pt]  (3,0) to[in=-90,out=165] (0.3,2);
\draw[line width=1.4pt]  (0.3,2) to[in=195,out=90] (3,4);

\draw[line width=1.4pt]  (0.3,2) -- (0.5,2.25);
\draw[line width=1.4pt]  (0.3,2) -- (0.1,2.25);
\draw[line width=1.4pt]  (2.7,2) -- (2.9,2.25);
\draw[line width=1.4pt]  (2.7,2) -- (2.5,2.25);
\end{tikzpicture}
\hspace{6em}
\begin{tikzpicture}[scale = 0.6]
\draw[line width=1.4pt]  (0,0) to[in=-90,out=15] (2.7,2);
\draw[line width=1.4pt]  (2.7,2) to[in=-15,out=90] (0,4);
\draw[line width=1.4pt]  (3,0) to[in=-90,out=165] (0.3,2);
\draw[line width=1.4pt]  (0.3,2) to[in=195,out=90] (3,4);

\draw[line width=1.4pt]  (0.3,2) -- (0.5,1.75);
\draw[line width=1.4pt]  (0.3,2) -- (0.1,1.75);
\draw[line width=1.4pt]  (2.7,2) -- (2.9,2.25);
\draw[line width=1.4pt]  (2.7,2) -- (2.5,2.25);
\end{tikzpicture}
\end{center}
\caption{Regular and irregular biangles.}\label{pic_biangles}
\end{figure}
\end{definition}

\begin{lemma}\label{lem_biangles}
Any angular diagram contains at most two irregular biangles. 
\end{lemma}
\begin{proof}
The smoothing double points in irregular biangle gives us the Seifert circle that bound the area which doesn't contains another points Seifert circles. Since collection of concentric circles has only two such circles, the number of irregular biangles cannot be greater than two.
\end{proof}

\begin{definition}
Applying $R_2$ to irregular biangle is said to be \emph{tightening}, the inverse move is said to be \emph{bending}.
\end{definition}

\section{Preliminary theorems}\label{sec-3}
According to Theorem \ref{th_kh2} any doodle has diagram with minimal number of double points. Moreover, the sequence of moves connecting any diagram of doodle with equivalent minimal diagram has an important property, it contains only moves that reduces number of double points. In this section we consider some other properties of this sequence.

\begin{theorem}\label{th_biangles}
Let $D$ be an angular diagram and $D'$ be the result of sequence of tightenings of $D$. If $D$ doesn't contain regular biangles, then $D'$ doesn't contain regular biangles.
\end{theorem}
\begin{proof}
Suppose the regular biangle arises after applying $k$ tightenings for some $k \geq 1$. The fragment of the diagram to which these transformations can be applied is uniquely determined. This fragment before (on the left) and after smoothing (on the right) is illustrated in Figure \ref{pic_k-bendings}.
\begin{figure}[H]
\begin{center}
\begin{tikzpicture}[scale = 0.3]
\draw[line width=1.4pt]  plot [smooth, tension=0.7] coordinates
{(0,1)(9,2)(14,5)(14,9)(9,12)(0,13)};

\draw[line width=1.4pt]  (0,7) to[out=90, in=-115] (2,14);
\draw[line width=1.4pt]  (2,0) to[out=115, in=-90] (0,7);
\draw[line width=1.4pt]  (0,7) -- (-0.3,7.3);
\draw[line width=1.4pt]  (0,7) -- (0.3,7.3);

\draw[line width=1.4pt]  (2,7) to[out=90, in=-115] (4,14);
\draw[line width=1.4pt]  (4,0) to[out=115, in=-90] (2,7);
\draw[line width=1.4pt]  (2,7) -- (1.7,6.7);
\draw[line width=1.4pt]  (2,7) -- (2.3,6.7);

\draw[line width=1.4pt]  (10,7) to[out=90, in=-115] (12,14);
\draw[line width=1.4pt]  (12,0) to[out=115, in=-90] (10,7);
\draw[line width=1.4pt]  (10,7) -- (9.7,6.7);
\draw[line width=1.4pt]  (10,7) -- (10.3,6.7);

\draw[line width=1.4pt] (14,5) -- (14.4,5);
\draw[line width=1.4pt] (14,5) -- (13.92,5.4);
\draw (6,7) node {\huge\ldots};
\end{tikzpicture}
\hspace{4em}
\begin{tikzpicture}[scale = 0.3]
\draw[line width=1.4pt,gray!30] plot [smooth, tension=0.7] coordinates
{(0,1)(9,2)(14,5)(14,9)(9,12)(0,13)};

\draw[line width=1.4pt,gray!30] (0,7) to[out=90, in=-115] (2,14);
\draw[line width=1.4pt,gray!30] (2,0) to[out=115, in=-90] (0,7);
\draw[line width=1.4pt] (0,7) -- (-0.2,7.2);
\draw[line width=1.4pt] (0,7) -- (0.2,7.2);

\draw[line width=1.4pt,gray!30] (2,7) to[out=90, in=-115] (4,14);
\draw[line width=1.4pt,gray!30] (4,0) to[out=115, in=-90] (2,7);
\draw[line width=1.4pt] (2,7) -- (1.8,6.8);
\draw[line width=1.4pt] (2,7) -- (2.2,6.8);

\draw[line width=1.4pt,gray!30] (10,7) to[out=90, in=-115] (12,14);
\draw[line width=1.4pt,gray!30] (12,0) to[out=115, in=-90] (10,7);

\draw[line width=1.4pt]  plot [smooth cycle, tension = 0.7] coordinates
{(11,3)(14,5)(14,9)(11,11)(10,7)};
\draw[line width=1.4pt]  (10,7) -- (9.7,6.7);
\draw[line width=1.4pt]  (10,7) -- (10.3,6.7);

\draw[line width=1.4pt] plot [smooth, tension=0.8] coordinates
{(12,0)(11,2)(9,2)(8,7)(9,12)(11,12)(12,14)};
\draw[line width=1.4pt]  (8,7) -- (7.7,6.7);
\draw[line width=1.4pt]  (8,7) -- (8.3,6.7);

\draw[line width=1.4pt] plot [smooth, tension=0.7] coordinates
{(6,0)(5,2)(3,2)(2,7)(3,12)(5,12)(6,14)};
\draw[line width=1.4pt]  (2,7) -- (1.7,6.7);
\draw[line width=1.4pt]  (2,7) -- (2.3,6.7);

\draw[line width=1.4pt] plot [smooth, tension=0.7] coordinates
{(4,0)(3.5,1.1)(2,1)(2,0)};
\draw[line width=1.4pt]  (2,1) -- (2.1,1.3);
\draw[line width=1.4pt]  (2,1) -- (2.3,0.9);

\draw[line width=1.4pt] plot [smooth, tension=0.7] coordinates
{(4,14)(3.5,12.9)(2,13)(2,14)};
\draw[line width=1.4pt]  (2,13) -- (2.1,12.7);
\draw[line width=1.4pt]  (2,13) -- (2.3,13.1);

\draw[line width=1.4pt] plot [smooth, tension=0.7] coordinates
{(0,0.5)(1.5,2)(0,7)(1.5,12)(0,13.5)};
\draw[line width=1.4pt]  (0,7) -- (-0.3,7.3);
\draw[line width=1.4pt]  (0,7) -- (0.3,7.3);
\draw (6,7) node {\huge\ldots};
\end{tikzpicture}
\end{center}
\caption{Considered fragment and its smoothing.}\label{pic_k-bendings}
\end{figure}
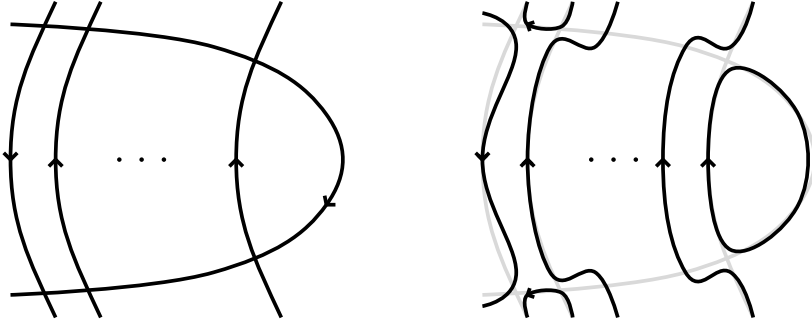
Let $S(D)$ be the result of smoothing of all double points of diagram $D$. Obviously, if $S(D)$ contains fragment described in Figure \ref{pic_k-bendings} on the left, then $S(D)$ contains the fragment described in Figure \ref{pic_k-bendings} on the right. So $S(D)$ cannot be the collection of concentric circles, it is contradict the assumption that $D$ is annular and Lemma \ref{smooth-lemma}.
\end{proof}

Previous Theorem \ref{th_biangles} shows that we can start the sequence from removing all regular biangles then it continue by tightenings before first move $R_1$.

\begin{definition}
The fragment of diagram, shown in Figure \ref{gen_biangle}, is called the generalized biangle.
\begin{figure}[H]
\begin{center}
\begin{tikzpicture}[scale = 0.65]
\draw[line width=1.4pt]  plot [smooth, tension=0.7] coordinates
{(0,0)(1.8,1.3)(1.8,2.7)(0,4)};
\draw[line width=1.4pt]  plot [smooth, tension=0.7] coordinates
{(2,0)(0.2,1.3)(0.2,2.7)(2,4)};

\draw [fill = white] (2-0.4,2-0.3) rectangle (2+0.4,2+0.3);
\draw [fill = white] (0-0.4,2-0.3) rectangle (0+0.4,2+0.3);
\draw (2,2) node {$l$};
\draw (0,2) node {$k$};
\end{tikzpicture}
\end{center}
\caption{Generalized biangle.}\label{gen_biangle}
\end{figure}
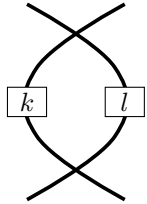
Here mark on ark indicates the number of non-intersecting ``parallel'' arcs that are oriented identically.
\end{definition}

Generalized biangle can be obtained from collection of parallel arcs by different sequences of bendings but we will consider this sequences as special move. Moreover, as we are interested in annular diagrams we will consider special case of sequences defined as following.

\begin{definition}
The sequence of bendigs, creating generalized biangle, will be called a \emph{generalized bending}. We assume that the following conditions are satisfied:
\begin{enumerate}
\item any bending in the generalized bending is applying to arcs of different Seifert circles,
\item all possible bendigs applied to arcs of different Seifert circles are used.
\end{enumerate}
\begin{figure}[H]
\begin{center}
\begin{tikzpicture}[scale = 0.65]
\draw[line width=1.4pt]  (5.5,0) to[in=-90,out=15] (8.2,2);
\draw[line width=1.4pt]  (8.2,2) to[in=-15,out=90] (5.5,4);
\draw[line width=1.4pt]  (8.5,0) to[in=-90,out=165] (5.8,2);
\draw[line width=1.4pt]  (5.8,2) to[in=195,out=90] (8.5,4);

\draw[line width=1.4pt]  (8.2,2) -- (8,2.25);
\draw[line width=1.4pt]  (8.2,2) -- (8.4,2.25);

\draw[line width=1.4pt]  (5.8,2) -- (5.6,1.75);
\draw[line width=1.4pt]  (5.8,2) -- (6,1.75);

\draw [line width=1.7pt, ->] (3.5,2) -- (5, 2);

\draw[line width=1.4pt]  plot [smooth, tension=0.7] coordinates
{(0,0)(0.3,2)(0,4)};
\draw[line width=1.4pt]  (0.3,2) -- (0.1,2.25);
\draw[line width=1.4pt]  (0.3,2) -- (0.5,2.25);

\draw[line width=1.4pt]  plot [smooth, tension=0.7] coordinates
{(3,0)(2.7,2)(3,4)};
\draw[line width=1.4pt]  (2.7,2) -- (2.9,1.75);
\draw[line width=1.4pt]  (2.7,2) -- (2.5,1.75);

\draw [fill = white] (2.8-0.4,2.9-0.3) rectangle (2.8+0.4,2.9+0.3);
\draw [fill = white] (0.3-0.4,1.3-0.3) rectangle (0.3+0.4,1.3+0.3);
\draw (2.8,2.9) node {$k$};
\draw (0.3,1.3) node {$l$};

\draw [fill = white] (8-0.4,2.95-0.3) rectangle (8+0.4,2.95+0.3);
\draw [fill = white] (6-0.4,1.2-0.3) rectangle (6+0.4,1.2+0.3);
\draw (8,2.95) node {$l$};
\draw (6,1.2) node {$k$};
\end{tikzpicture}
\end{center}
\caption{Generalized bending.}\label{gen_bending}
\end{figure}
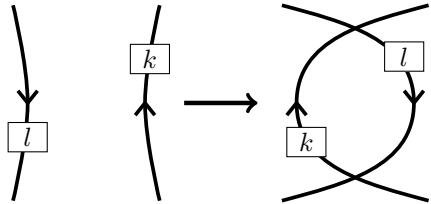
If some of conditions cannot be satisfied we say that generalized bending is not applicable.
The inverse to generalized bending is called generalized tightening.
\end{definition}

In a sense, generalized bending is maximal sequence of bending creating generalized biangle, so the following examples are designed to make this concept more clear.

\begin{example}
In Figure \ref{pic_not__max} we give an example of bending that is not generalized bending, because we have a possible bending for arcs $\alpha$ and $\gamma$.
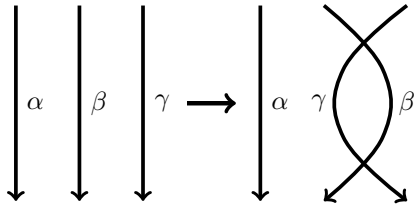
\begin{figure}[H]
\begin{center}
\begin{tikzpicture}[scale = 0.65]
\draw [line width=1.4pt, ->] (0,4) -- (0,0);
\draw [line width=1.4pt, ->] (1.3,4) -- (1.3,0);
\draw [line width=1.4pt, ->] (2.6,4) -- (2.6,0);
\draw (0.4,2) node {$\alpha$};
\draw (1.7,2) node {$\beta$};
\draw (3,2) node {$\gamma$};
\draw [line width=1.7pt, ->] (3.5,2) -- (4.5,2);
\draw [line width=1.4pt, ->] plot [smooth, tension=0.7] coordinates
{(8,4)(6.7,2.7)(6.7,1.3)(8,0)};
\draw [line width=1.4pt, ->] plot [smooth, tension=0.7] coordinates
{(6.3,4)(6.2+1.3,2.7)(6.2+1.3,1.3)(6.3,0)};
\draw [line width=1.4pt, ->] plot [smooth, tension=0.7] coordinates
{(5,4)(5,0)};
\draw (5.4,2) node {$\alpha$};
\draw (6.2,2) node {$\gamma$};
\draw (8,2) node {$\beta$};
\end{tikzpicture}
\end{center}
\caption{The bending which is not maximal.}\label{pic_not__max}
\end{figure}
\end{example}

\begin{example}
Set of three circles is an example of diagram to which generalized bending is not applicable. In Figure \ref{pic_2circles} we present a result of applying bending to circles $\alpha$ and $\beta$, but there is possible bending for circles $\alpha$ and $\gamma$. But result of applying two bendings is not generalized biangle.
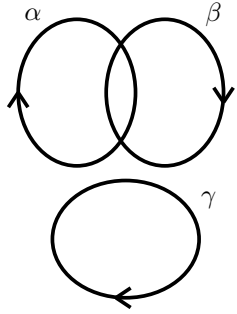
\begin{figure}[H]
\begin{center}
\begin{tikzpicture}[scale = 0.65]
\draw[line width=1.4pt] (1.2,4.2) circle [x radius=1.2cm, y radius=1.5cm];
\draw[line width=1.4pt] (3,4.2) circle [x radius=1.2cm, y radius=1.5cm];
\draw[line width=1.4pt] (2.2,1.2) circle [x radius=1.5cm, y radius=1.2cm];
\draw (0.3,5.8) node {$\alpha$};
\draw (4,5.8) node {$\beta$};
\draw (3.9,2) node {$\gamma$};

\draw [line width=1.4pt] (0,4.2) -- (-0.2,4.2-0.3);
\draw [line width=1.4pt] (0,4.2) -- (0.2,4.2-0.3);

\draw [line width=1.4pt] (4.2,4) -- (4.2-0.2,4.3);
\draw [line width=1.4pt] (4.2,4) -- (4.2+0.2,4.3);
\draw [line width=1.4pt] (2,0) -- (2+0.3,0.2);
\draw [line width=1.4pt] (2,0) -- (2+0.3,-0.2);

\end{tikzpicture}
\end{center}
\caption{Result of bending of two circles.}\label{pic_2circles}
\end{figure}
\end{example}

\begin{theorem}\label{th_circles_nuber}
Let $D$ be the annular diagram then generalized tightening doesn't change number of Seifert circles. 
\end{theorem}
\begin{proof}
Annular diagram with irregular generalized biangle is illustrated in Figure \ref{pic_tighted}, here $B_1$ and $B_2$ are some twin diagrams on $n$ strands. So after smoothing of all double points we have $n+k$ Seifert circles.
\begin{figure}[ht!]
\begin{center}
\begin{tikzpicture}[scale = 0.3]
\draw[line width=1.4pt] (0,12) rectangle (4,10);

\draw[line width=1.4pt] (3,10) to[in=90, out=-90] (5,8) to (5,6);
\draw[line width=1.4pt] (5,10) to[in=90, out=-90] (3,8);
\draw[line width=1.4pt] (0,6) rectangle (4,8);
\draw[line width=1.4pt] (3,6) to[in=90, out=-90] (5,4);
\draw[line width=1.4pt] (5,6) to[in=90, out=-90] (3,4);

\draw [line width=1.4pt] (5,10) to (5,12) 
                                to[out=90,in =180] (6,13)
                                to[out=0,in =90] (7,12) to (7,3)
                                to[out=-90,in =0] (6,2)
                                to[out=180,in =-90] (5,4);

\draw [line width=1.4pt] (3,12) to[out=90,in =180] (5.5,14)
								to[out=0,in=90] (8,12) to (8,3)
                                to[out=-90,in =0] (6,1)
                                to[out=180,in =-90] (3,3) to (3,4);

\draw [line width=1.4pt] (1,10) -- (1,8);
\draw [line width=1.4pt] (1,12) to[out=90,in =180] (5.5,15)
								to[out=0,in =90] (9,12) to (9,3)
                                to[out=-90,in =0] (5.5,0)
                                to[out=180,in =-90] (1,3) to (1,6);
                                
\draw (2,7) node {\large$B_2$};
\draw (2,11) node {\large$B_1$};

\draw [fill = white] (5-0.5,7-0.6) rectangle (5+0.5,7+0.6);
\draw (5,7) node {$k$};
\draw [fill = white] (7-0.5,7-0.6) rectangle (7+0.5,7+0.6);
\draw (7,7) node {$k$};
\end{tikzpicture}
\hspace{4em}
\begin{tikzpicture}[scale = 0.35]
\draw [line width=1.4pt] (0,10) rectangle (4,8);
\draw (2,9) node {\large$B_1$};
\draw [line width=1.4pt] (3,8) to[out=-90,in=180] (4,7)
						       to[out=0,in=-90] (5,8) to (5,10)
                               to[out=90,in=0] (4,11)
                               to[out=180,in=90] (3,10);

\draw [line width=1.4pt] (0,3) rectangle (4,5);
\draw (2,4) node {\large$B_2$};
\draw [line width=1.4pt] (3,3) to[out=-90,in=180] (4,2)
						       to[out=0,in=-90] (5,4) to (5,5)
                               to[out=90,in=0] (4,6)
                               to[out=180,in=90] (3,5);

\draw [line width=1.4pt] (1,5) -- (1,8);
\draw [line width=1.4pt] (1,3) to[out=-90,in=180] (4,0)
							   to[out=0,in=-90] (7,3) to (7,10)
                               to[out=90,in=0] (4,13)
                               to[out=180,in=90] (1,10);
                               
\draw [fill = white] (5-0.5,4-0.6) rectangle (5+0.5,4+0.6);
\draw [fill = white] (5-0.5,9-0.6) rectangle (5+0.5,9+0.6);
\draw (5,4) node {$k$};
\draw (5,9) node {$k$};
\end{tikzpicture}
\end{center}
\caption{Equivalent diagram before and after tightening.}\label{pic_tighted}
\end{figure}
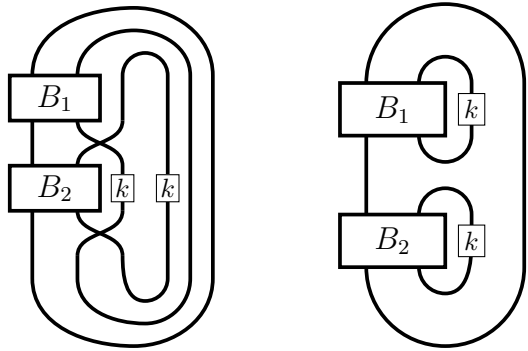
We have $n-k$ Seifert circles which are not involved tightening and $2k$ Seifert circles obtained after it.
\end{proof}

\begin{theorem}\label{th_bend_seq}
Let $D$ be an annular diagram related with minimal diagram by sequence of bendings $B_1, B_2, \ldots, B_k$. Then this sequence can be decomposed into at most two generalized bending.
\end{theorem}
\begin{proof}
It is easy to see that bending applied to different Seifert circles doesn't change number of Seifert circles, and bending applied to arcs of the same Seifert circle increases number of Seifert circles. 
According to the Theorem \ref{th_circles_nuber} the number of Seifert circles of $D$ and corresponding minimal diagram are equal, so all $B_i$ are applying to different Seifert circles. All arcs of annular diagram have the same orientation so there is no possible bendings that mean generalized bending is applicable and both of conditions are satisfied.

Any generalized bending creates exact one irregular biangle, so number according the Lemma \ref{lem_biangles} of generalized bendings cannot be greater than two.
\end{proof}
If we consider the sequence of inverses of bendings $B_1, B_2, \ldots, B_k$ from condition of previous theorem, we get the same decomposition for sequence of tightenings.

\begin{theorem}\label{th_gen_bend_uniq}
For any collection of arcs the generalized bending is uniquely determined.
\end{theorem}
\begin{proof}
Let $\alpha$, $\beta$ and $\gamma$ be collections of equally oriented arcs. Suppose that there are two generalized bendings for $\alpha$ with $\beta$ and $\gamma$. Apply generalized bendings for $\alpha$ and $\beta$, see Figure \ref{gen_bend}.

\begin{figure}[ht!]
\begin{center}
\begin{tikzpicture}[scale = 0.7]
\draw[line width=1.4pt] plot [smooth, tension=0.7] coordinates
{(0,1)(0.5,3)(0,5)};
\draw[line width=1.4pt] plot [smooth, tension=0.7] coordinates
{(3.5,1)(3,3)(3.5,5)};
\draw[line width=1.4pt] plot [smooth, tension=0.7] coordinates
{(0,0)(2,0.5)(4,0)};

\draw[line width=1.4pt] (2,0.5) -- (2-0.2,0.5+0.15);
\draw[line width=1.4pt] (2,0.5) -- (2-0.2,0.5-0.15);
\draw[line width=1.4pt] (0.5,3) -- (0.5-0.2,3+0.3);
\draw[line width=1.4pt] (0.5,3) -- (0.5+0.15,3+0.3);
\draw[line width=1.4pt] (3,3) -- (3-0.15,3-0.3);
\draw[line width=1.4pt] (3,3) -- (3+0.2,3-0.3);

\draw (0.7,4.5) node {$\alpha$};
\draw (3.7,4.5) node {$\beta$};
\draw (2,1) node {$\gamma$};

\draw[line width=1.4pt] plot [smooth, tension=1.5] coordinates
{(6,1)(9,3)(6,5)};
\draw[line width=1.4pt] plot [smooth, tension=1.5] coordinates
{(9.5,1)(6.5,3)(9.5,5)};
\draw[line width=1.4pt] plot [smooth, tension=0.7] coordinates
{(6,0)(8,0.5)(10,0)};

\draw[line width=1.4pt] (8,0.5) -- (8-0.2,0.5+0.15);
\draw[line width=1.4pt] (8,0.5) -- (8-0.2,0.5-0.15);

\draw[line width=1.4pt] (6.5,3) -- (6.5-0.17,3-0.3);
\draw[line width=1.4pt] (6.5,3) -- (6.5+0.2,3-0.3);
\draw[line width=1.4pt] (9,3) -- (9-0.17,3+0.3);
\draw[line width=1.4pt] (9,3) -- (9+0.2,3+0.3);

\draw (6.5,4.5) node {$\alpha$};
\draw (9.5,4.5) node {$\beta$};
\draw (8,0.9) node {$\gamma$};
\draw [line width=1.7pt, ->] (4.5,2.5) -- (5.5, 2.5);
\end{tikzpicture}
\end{center}
\caption{Applying one of two possible bendings.}\label{gen_bend}
\end{figure}
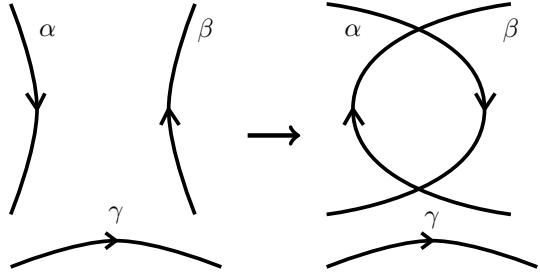
After this bending for $\beta$ and $\gamma$ move $R_2$ remains applicable, if generalized bending for $\alpha$ and $\beta$ is applicable $\gamma$ has arcs of the same Seifert circles as $\beta$. So generalized bending for $\alpha$ and $\gamma$ is generalized bending for $\alpha$ and $\beta$.
\end{proof}

\begin{theorem}\label{th_uniq_annular}
Let $D_1$ and $D_2$ be the equivalent annular diagrams related with the same minimal digram by generalized bendings, then $D_1$ and $D_2$ are related by circles shift.
\end{theorem}
\begin{proof}
Let $D_0$ be the minimal diagram corresponding to $D_1$ and $D_2$. As $D_1$ and $D_2$ are results of application the generalized bendings to arcs of $D_0$, according to the Theorem \ref{th_gen_bend_uniq} for any collections of arcs the generalized bending applies uniquely, so $D_1$ and $D_2$ will be equivalent up circles shift(cause minimal diagram defined up to circles shift).
\end{proof}
It is obvious that result of tightening of annular diagram is not annular diagram. Annular diagram related with minimal digram by generalized bending has minimal number of double points in class of annular diagrams so we will call such diagrams \emph{minimal annular} diagram. The previous Theorem \ref{th_uniq_annular} shows that minimal annular diagram is unique up to circles shift.

\section{Twins equivalence and Markov theorem}\label{sec-4}

\begin{definition}
For any $n$ and $\beta \in TW_n$, $I \in TW_1$, define the following moves:
\begin{itemize}
\addtolength{\itemindent}{5mm}
\item [$M_1:$] $\beta \otimes I \leftrightarrow I \otimes \beta$,
\item [$M_2:$] $\beta \rightarrow \alpha \beta \alpha^{-1}$,
\item [$M_3:$] $\beta \rightarrow (I \otimes \beta) 
					s_1 s_2 \ldots s_{i-1} s_i s_{i-1} \ldots s_2 s_1$,
\item [$M_4:$] $\beta \rightarrow (\beta \otimes I) 
					s_n s_{n-1} \ldots s_{i+1} s_i s_{i+1} \ldots s_{n-1} s_n$,
\end{itemize}
here $\alpha \in TW_n$, $s_i \in TW_{n+1}, i = 1,\ldots, n$.
\\
Two twins are \emph{$M$-equivalent} if they are related by sequence of moves $M_1, M_2, M_3, M_4$ and its inverses.
\end{definition}
The proof of following lemma is obvious.
\begin{lemma}\label{th_minimal_twin}
Let $D$ be minimal annular diagram that corresponds to twins $\beta$ and $\alpha$. Then $\beta$ and $\alpha$ are equivalent up to moves $M_1$ and $M_2$.
\end{lemma}

\begin{lemma}
Let $\beta$ be a twin then $\widehat{M_i(\beta)} = \widehat{\beta}$ for any $i = 1,2,3,4$.
\end{lemma}
\begin{proof}
Obviously, it is only worth mentioning that the moves $M_3$ and $M_4$ are the composition of $R_1$ and some number of bendings, $M_1$ correspond to circles shift.
\end{proof}

\begin{theorem}\label{th_markov}
Any two twins with equivalent closures are $M$-equivalent.
\end{theorem}
\begin{proof}
We will prove that any two twins with equivalent closures share a common twin in $M$-equivalence classes.

Consider some twin $\alpha$. Let $d$ be the twin diagram such that $\widehat{d}$ is minimal annular diagram of $\widehat{\alpha}$. Let $\Delta$ be the twin presented by $d$. According to the Theorem \ref{th_uniq_annular} the minimal annular diagram is unique up to circles shift thus $\Delta$ doesn't depend on the choice $\alpha$. We will prove that twins $\alpha$ and $\Delta$ are $M$-equivalent.

Fix some diagram $a$ of $\alpha$. Define $\big| a\big|$ as the difference between the number of double points of $a$ and $d$.

\begin{enumerate}
\item Base case.\\
If $\big|a\big| = 0$, according the Lemma \ref{th_minimal_twin} $\alpha$ related with $\Delta$ by moves $M_1$ and $M_2$.
Assume theorem statement holds for $\big|a\big| \le N$.
\item Inductive step.\\
Assume $\big|a\big| = N+1$.

If $\widehat{a}$ contains a regular beangle, we can apply $R_2$ to it. 
Obtained diagram is annular and obtained by equivalence of twin diagrams and $M_2$, that reduces number of double points.\\

If $\widehat{\alpha}$ contains a loop, applying the inverses $M_3$ or $M_4$ reduce the number of double points and diagram remains the to be annular.\\

Suppose $\widehat{\alpha}$ doesn't contains a loop and regular beangles.

Consider sequence $R^1, \ldots, R^k$ of moves $R_1, R_2$, relating $\widehat{\alpha}$ with minimal diagram.
Let $j = 2,\ldots, k$, such that $R^j$ is first move of type $R_1$, according the Theorem \ref{th_biangles} all $R^1, \ldots, R^{j-1}$ are tightenings, moreover, according to Theorem \ref{th_bend_seq} we can assume this sequence as generalized tigthening. Thus composition of $R^1, \ldots, R^j, R^{j+1}$ is inverse for move $M_3$ or $M_4$ and that decreases number of double points.

Number of double points decreased by moves $M_2, M_3, M_4$ so induction assumption is satisfied.

Such $j$ always exists, in other case we have a sequence of tightenings relating $\widehat{a}$ with minimal diagram that means $\widehat{a}$ is minimal annular diagram and $\big|a\big| = 0$.
\end{enumerate}Two twins with equivalent closures are $M$-equivalent with the common twin so they are $M$-equivalent.
\end{proof}


\begin{thebibliography}{0}
\bibitem{1}
J.~Birman, Knots, links, and mapping class groups,{\it Annals of Math Study}, {\bf 82}, Princeton University Press (1974)

\bibitem{2}
R.~Fenn, P.~Taylor, Introducing doodles, {\it Lect. Notes in Math.}, {\bf 722}, (Springer, Berlin, 1979), pp.~37--43

\bibitem{3}
M.~Khovanov, Doodle groups,
{\it Trans. Amer. Math. Soc.}, {\bf 349} (1997), 2297-2315

\bibitem{4}
A.~Bartholomew, R.~ Fenn, N.~Kamada, S.~Kamada,R.~Fenn, Doodles on surfaces I: An introduction to their basic properties, preprint: arXiv:1612.08473v1

\bibitem{5}
K.~Gongopadhyay, S.~Dey, Commutator subgroups of twin groups and Grothendieck's cartographical groups, preprint: arXiv:1804.05375v2
\end{thebibliography}
\end{document}